\documentclass[12pt,a4paper,leqno]{amsart}

\usepackage{amssymb,amsmath,latexsym,amsthm,enumerate}

\theoremstyle{thmit} 
\newtheorem{thm}{Theorem}[section]
\newtheorem{lem}[thm]{Lemma}

\newtheorem{prop}[thm]{Proposition}
\newtheorem{conj}[thm]{Conjecture}

\usepackage{mathrsfs}

\theoremstyle{definition}
\newtheorem{remark}[thm]{Remark}

\makeatletter
\@namedef{subjclassname@2020}{%
  \textup{2020} Mathematics Subject Classification}
\makeatother

\numberwithin{equation}{section}

\allowdisplaybreaks

\begin{document} 

\title[A function-field analogue]{A function-field 
analogue of the Goldbach counting function and the associated Dirichlet series}
\author{Shigeki Egami}
\address{S. Egami: 1-2-704 Naito-machi, Shinjuku-ku, 
Tokyo 160-0014, Japan}
\email{egami-sg@shibaura-it.ac.jp}
\author {Kohji Matsumoto}
\address{K. Matsumoto: Graduate School of Mathematics, 
Nagoya University, 
Chikusa-ku, Nagoya 464-8602, Japan}
\email{kohjimat@math.nagoya-u.ac.jp}


\thanks{This work was supported by Japan Society for the Promotion of Science, Grant-in-Aid for Scientific Research No. 22K03267 (K. Matsumoto).}

\maketitle

\begin{abstract}
We consider a function-field analogue of Dirichlet series associated with the Goldbach
counting function, and prove that it can, or cannot, be continued meromorphically to
the whole plane.    When it cannot, we further prove the existence of the natural
boundary of it.
\end{abstract}

\bigskip
\section{Introduction}\label{sec-1}

Let $\mathbb{N}$ be the set of positive integers, 
$\mathbb{N}_0=\mathbb{N}\cup\{0\}$,
$\mathbb{Z}$ the set of rational integers,
$\mathbb{Q}$ the set of rational
numbers, $\mathbb{R}$ the set of real numbers,
and $\mathbb{C}$ the set of complex numbers.

In their study \cite{HL24} on Goldbach's problem, Hardy and Littlewood considered
the function
\begin{align}\label{G_2-def}
G_2(n)=\sum_{\substack{k,m\in\mathbb{N} \\k+m=n}}
\Lambda(k)\Lambda(m) \qquad (n\in\mathbb{N}),
\end{align}
where $\Lambda(\cdot)$ denotes the von Mangoldt function.
Later, Fujii \cite{Fu91} studied the average $\sum_{n\leq X}G_2(n)$ extensively.

Inspired by Fujii's work, the authors \cite{EM07} considered the Dirichlet series
\begin{align}\label{Phi_2-def}
\Phi_2(s)=\sum_{n=1}^{\infty}\frac{G_2(n)}{n^s},
\end{align}
which is convergent absolutely in the region $\Re s>2$.
In \cite{EM07} it is proved that, under the Riemann hypothesis (RH) for the Riemann
zeta-function $\zeta(s)$, $\Phi_2(s)$ can be continued meromorphically to
$\Re s>1$.   Then the authors raised the following
\begin{conj}\label{conj-EM}
The line $\Re s=1$ is the natural boundary of $\Phi_2(s)$.
\end{conj}

Let $\mathcal{I}$ be the set of all imaginary parts of non-trivial zeros of $\zeta(s)$.
A well-known conjecture predicts that the positive elements of $\mathcal{I}$ would be
linearly independent over $\mathbb{Q}$.   The following statement is a special case of
this conjecture:

(A) If $\gamma_j\in\mathcal{I}$ ($1\leq j\leq 4$) and $\gamma_1+\gamma_2=
\gamma_3+\gamma_4(\neq 0)$, then $(\gamma_3,\gamma_4)$ equals $(\gamma_1,\gamma_2)$ or
$(\gamma_2,\gamma_1)$.

In \cite{EM07}, the authors introduced the following quantitative version of
Conjecture (A):

(B) There exists a constant $\alpha$, with $0<\alpha<\pi/2$, such that if 
$\gamma_j\in\mathcal{I}$ ($1\leq j\leq 4$), $\gamma_1+\gamma_2\neq 0$, and
$(\gamma_3,\gamma_4)$ is neither equal to $(\gamma_1,\Gamma_2)$ nor to
$(\gamma_2,\gamma_1)$, then
\begin{align}\label{ineq-B}
|(\gamma_1+\gamma_2)-(\gamma_3+\gamma_4)|\geq
\exp(-\alpha(|\gamma_1|+|\gamma_2|+|\gamma_2|+|\gamma_4|))
\end{align}
holds.

In \cite{EM07} the authors showed that Conjecture \ref{conj-EM} can be deduced if we
assume the RH and Conjecture (B).
Bhowmik and Schlage-Puchta \cite{BSP11} proved that Conjecture (B) is actually not necessary
to deduce Conjecture \ref{conj-EM}; it is enough to assume the RH and Conjecture (A).

The aim of the present note is to consider the function-field analogue of the above
results.   
The situation becomes much simpler, so this is a kind of ``toy model''. 
However it is still interesting because, in the function-field case, we can give an
{\it unconditional proof} of the existence or non-existence of the natural boundary.

\section {statement of results}\label{sec-2}

Let $\mathbb{F}_i$ ($i=1,2$) be finite fields whose cardinalities are 
$q_i=p_i^{r_i}$ with primes $p_1,p_2$.
Let $K_i$ be function fields of one variable over $\mathbb{F}_i$.
The zeta-function of $K_i$ is defined by
\begin{align}\label{zeta-def}
\zeta_{K_i}(s)=\sum_{A_i}(NA_i)^{-s}=\prod_{P_i} (1-(NP_i)^{-s})^{-1},
\end{align}
where $s$ is a complex variable, $A_i$ runs over all effective divisors of $K_i$,
$P_i$ runs over all primes of $K_i$, 
and $NA_i=q_i^{\deg A_i}$.
The above series expression can be rewritten as $\sum_{n=1}^{\infty}b_i(n)q_i^{-ns}$,
where $b_i(n)$ denotes the number of effective divisors of degree $n$ in $K_i$. 
Therefore \eqref{zeta-def} is convergent absolutely in the domain $\Re s>1$, because
\begin{align}\label{est-b}
b_i(n)=O(q_i^n)
\end{align}
holds (see \cite[p.52]{Ro02}).

It is well known that the above zeta-function can be written as 
$\zeta_{K_i}(s)=Z_{K_i}(q_i^{-s})$, with (putting $u_i=u_i(s)=q_i^{-s}$)
\begin{align}\label{expression}
Z_{K_i}(u_i)=\frac{L_{K_i}(u_i)}{(1-u_i)(1-q_i u_i)},\quad
L_{K_i}(u_i)=\prod_{j=1}^{2g_i}(1-\pi_{ji}u_i),
\end{align}
where $g_i$ is the genus of $K_i$ and $\pi_{ji}\in\mathbb{C}$, $|\pi_{ji}|=q_i^{1/2}$
(see, for example, \cite[Theorems 5.9 and 5.10]{Ro02}).

From \eqref{zeta-def} it easily follows that
\begin{align}\label{log-der}
\frac{\zeta_{K_i}^{\prime}}{\zeta_{K_i}}(s)=-\sum_{P_i}\sum_{h=1}^{\infty}
\frac{\log (NP_i)}{(NP_i)^{hs}}
=-\sum_{A_i} \frac{\Lambda_{K_i}(A_i)}{(NA_i)^s},
\end{align}
where $\Lambda_{K_i}(A_i)$ is the
function-field analogue of the von Mangoldt function defined by
\begin{align}\label{vonMangoldt}
\Lambda_{K_i}(A_i)=\begin{cases}
\log(NP_i) & {\rm if} \; A_i=h P_i \;{\rm with}\;{\rm some}\; P_i,h\\
0 & {\rm otherwise}.
\end{cases}
\end{align}
Using \eqref{vonMangoldt} we now define the analogue of \eqref{G_2-def} by
\begin{align}\label{G-def}
G_{2}^{FF}(n)=\sum_{\substack{A_1,A_2 \\ NA_1+NA_2=n}}\Lambda_{K_1}(A_1)\Lambda_{K_2}(A_2).
\end{align}
This function may be regarded as the counting function of an analogue of the Goldbach
problem in function fields, that is, how frequently $n\in\mathbb{N}$ can be written as
$n=NP_1+NP_2$, where $P_i$ are primes of $K_i$ ($i=1,2$).   The Dirichlet series 
associated with \eqref{G-def} is
\begin{align}\label{DS-def}
\Phi_2^{FF}(s)=\sum_{n=1}^{\infty}\frac{G_2^{FF}(n)}{n^s}.
\end{align}

We first notice the following
\begin{prop}\label{prop-1}
The Dirichlet series \eqref{DS-def} is convergent absolutely in the region
$\Re s>2$.
\end{prop}

The main results in the present article are as follows.

\begin{thm}\label{thm-1}
If $p_1=p_2$, then $\Phi_2^{FF}(s)$ can be continued meromorphically to the whole
plane $\mathbb{C}$.
\end{thm}

\begin{thm}\label{thm-2}
If $p_1\neq p_2$, then the vertical line $\Re s=2$ is the natural boundary of
$\Phi_2^{FF}(s)$.
\end{thm}

First we will show Proposition \ref{prop-1} in Section \ref{sec-3}.   Then,
after some preparations in Section \ref{sec-4}, we will prove the main theorems
in Section \ref{sec-5}.

\section{Proof of Proposition \ref{prop-1}}\label{sec-3}

By the definition \eqref{vonMangoldt} it is clear that 
$\Lambda_{K_1}(A_1)\Lambda_{K_2}(A_2)\neq 0$ only
when $A_1=h_1P_1$ and $A_2=h_2P_2$ with $h_1,h_2\in\mathbb{N}$ and primes $P_1,P_2$. 
Therefore
\begin{align}\label{3-1}
&\Phi_2^{FF}(s)=\sum_{P_1,P_2}\sum_{h_1,h_2=1}^{\infty}
\frac{\log(NP_1)\log(NP_2)}{(N(h_1P_1)+N(h_2P_2))^s}\\
&=\sum_{P_1,P_2}\sum_{h_1,h_2=1}^{\infty}
\frac{\deg P_1 \log q_1\cdot \deg P_2 \log q_2}{(q_1^{\deg (h_1P_1)}
+q_2^{\deg (h_2P_2)})^s}\notag\\
&\leq \sum_{k_1,k_2=1}^{\infty}\sum_{\substack{P_1,P_2,h_1,h_2 \\ \deg(h_1P_1)=k_1\\
\deg(h_2P_2)=k_2}}\frac{k_1\log q_1\cdot k_2\log q_2}
{(q_1^{k_1}+q_2^{k_2})^s}.\notag
\end{align}
The number of the pairs $(h_i,P_i)$ ($i=1,2$) satisfying $\deg(h_i P_i)=k_i$
is $O(q_i^{k_i})$ by \eqref{est-b}.   Therefore the right-hand side of \eqref{3-1}
is
\begin{align}\label{3-2}
\ll_{q_1,q_2}\sum_{k_1,k_2=1}^{\infty}\frac{k_1 q_1^{k_1}\cdot k_2 q_2^{k_2}}
{(q_1^{k_1}+q_2^{k_2})^{\Re s}}
\leq \sum_{k_1=1}^{\infty}\frac{k_1}{ q_1^{k_1(\Re s/2-1)}}
\sum_{k_2=1}^{\infty}\frac{k_2}{ q_2^{k_2(\Re s/2-1)}},
\end{align}
which converges for $\Re s>2$.   The assertion of Proposition \ref{prop-1} follows.

\section {An integral expression}\label{sec-4}

We first recall the information on the distribution of zeros and poles of $\zeta_{K_i}(s)$.
From the expression \eqref{expression} we find that the poles of $\zeta_{K_i}(s)$
are given by $u_i=1$ and $q_iu_i=1$, that is, they can be written as
\begin{align}\label{pole}
\rho(a_i)=\frac{2a_i\pi \sqrt{-1}}{\log q_i},\quad 
\rho(b_i)=1+\frac{2b_i\pi\sqrt{-1}}{\log q_i}
\quad (a_i, b_i\in\mathbb{Z}).
\end{align}
These are not cancelled by the numerator, hence indeed poles and they are all simple; 
for $a_i=0$ and $b_i=0$ this fact is in \cite[Theorem 5.9]{Ro02}, and 
then the fact for other values of $a_i$ and $b_i$ follows because $\zeta_{K_i}(s)$ is 
periodic in $\Im s$ with period $2\pi\sqrt{-1}/\log q_i$.
(As for the residue at $s=1$, see \cite[p.309, formula (4)]{Ro02}.)

On the other hand, the zeros are given by $\pi_{ji}u_i=1$.    Since $|\pi_{ji}|=q_i^{1/2}$, 
we may write $\pi_{ji}=q_i^{1/2}\exp(\sqrt{-1}\arg\pi_{ji})$.
Then a zero $s$ should be a solution of
$$
q_i^{1/2}\exp(\sqrt{-1}\arg\pi_{ji}-s\log q_i)=1,
$$
and hence zeros can be written as
\begin{align}\label{zero}
\rho(c_{ji})=\frac{1}{2}+\frac{\sqrt{-1}}{\log q_i}(\arg\pi_{ji}+2c_{ji}\pi) 
\quad (c_{ji}\in\mathbb{Z},
1\leq j\leq 2g_i).
\end{align}


Differentiating \eqref{expression} we find that
\begin{align}\label{diff}
&\frac{\zeta_{K_i}^{\prime}}{\zeta_{K_i}}(s)
=-u_i\log q_i\\
&\quad\times\frac{(dL_{K_i}(u_i)/du_i)(1-u_i)(1-q_iu_i)+L_{K_i}(u_i)(1+q_i-2q_i u_i)}
{L_{K_i}(u_i)(1-u_i)(1-q_i u_i)},\notag
\end{align} 
hence all the poles and zeros of $\zeta_{K_i}(s)$ listed above are the simple poles of
$\zeta_{K_i}^{\prime}/\zeta_{K_i}(s)$.

Now we show an integral expression of $\Phi_2^{FF}(s)$.    First assume 
$\Re s>2+2\varepsilon$ with a small $\varepsilon>0$.
Then
\begin{align}\label{4-1}
&\Phi_2^{FF}(s)=\sum_{n=1}^{\infty}\frac{1}{n^s}
\sum_{\substack{A_1,A_2 \\ NA_1+NA_2=n}}\Lambda_{K_1}(A_1)\Lambda_{K_2}(A_2)\\
&=\sum_{A_1,A_2}\frac{\Lambda_{K_1}(A_1)\Lambda_{K_2}(A_2)}{(NA_1+NA_2)^s}
=\sum_{A_1,A_2}\frac{\Lambda_{K_1}(A_1)\Lambda_{K_2}(A_2)}{(NA_1)^s}
\left(1+\frac{NA_2}{NA_1}\right)^{-s}.\notag
\end{align}
Here we quote the Mellin-Barnes integral formula
\begin{align}\label{MB}
(1+\lambda)^{-s}=\frac{1}{2\pi\sqrt{-1}}\int_{(\alpha)}\frac{\Gamma(s-z)\Gamma(z)}
{\Gamma(s)}\lambda^{-z}dz,
\end{align}
where $s,\lambda\in\mathbb{C}$, $\lambda\neq 0$, $|\arg\lambda|<\pi$,
$\Re s>0$, $0<\alpha<\Re s$, and the path of integration is the vertical line from
$\alpha-\sqrt{-1}\infty$ to $\alpha+\sqrt{-1}\infty$.
Using this formula with $\lambda=NA_2/NA_1$, we find that the right-hand side of
\eqref{4-1} is
\begin{align*}
&=\sum_{A_1,A_2}\frac{\Lambda_{K_1}(A_1)\Lambda_{K_2}(A_2)}{(NA_1)^s}
\frac{1}{2\pi\sqrt{-1}}\int_{(\alpha)}\frac{\Gamma(s-z)\Gamma(z)}
{\Gamma(s)}\left(\frac{NA_2}{NA_1}\right)^{-z}dz\\
&=\frac{1}{2\pi\sqrt{-1}}\int_{(\alpha)}\frac{\Gamma(s-z)\Gamma(z)}{\Gamma(s)}
\sum_{A_1}\frac{\Lambda_{K_1}(A_1)}{(NA_1)^{s-z}}
\sum_{A_2}\frac{\Lambda_{K_2}(A_2)}{(NA_2)^z}dz.
\end{align*}
The absolute convergence of two infinite series on the right-hand side can be
verified if we choose $\alpha=1+\varepsilon$.    Then we have
\begin{align}\label{4-2}
\Phi_2^{FF}(s)=\frac{1}{2\pi\sqrt{-1}}\int_{(\alpha)}\frac{\Gamma(s-z)\Gamma(z)}{\Gamma(s)}
\frac{\zeta_{K_1}^{\prime}}{\zeta_{K_1}}(s-z)\frac{\zeta_{K_2}^{\prime}}
{\zeta_{K_2}}(z)dz.
\end{align}

Let $N\in\mathbb{N}$, and we shift the path of integration to $\Re z=-N+\varepsilon$.
Since $\zeta_{K_i}^{\prime}/\zeta_{K_i}$ is periodic with respect to $\Im s$, in view of
Stirling's formula this shifting is possible.
The relevant poles of the integrand are 

(A) $z=0,-1,-2,\ldots,-(N-1)$ coming from $\Gamma(z)$, 

(B)
$z=\rho(a_2), \rho(b_2), \rho(c_{j2})$ ($a_2,b_2,c_{j2}\in\mathbb{Z},
1\leq j\leq 2g_2$) coming from $(\zeta_{K_2}^{\prime}/\zeta_{K_2})(z)$.

The point $z=0$ is included in both (A) and (B).     Therefore $z=0$ is a double pole,
and its residue is
$$
R_0(s)=\frac{\Gamma'}{\Gamma}(s)-\Gamma'(1)
+\frac{(\zeta_{K_1}^{\prime}/\zeta_{K_1})'}{\zeta_{K_1}^{\prime}/\zeta_{K_1}}(s)
+\frac{C_{0,K_2}}{C_{-1,K_2}},
$$
where $C_{0,K_2},C_{-1,K_2}$ are determined by the Laurent expansion
$$
\zeta_{K_2}(s)=\frac{C_{-1,K_2}}{s}+C_{0,K_2}+\cdots
$$
at $s=0$.
All other poles listed in (A) and (B) are simple.   Therefore, shifting the path 
and counting the residues we obtain
\begin{align}\label{4-3}
\Phi_2^{FF}(s)=\Sigma_1(s)+\Sigma_{1/2}(s)+\Sigma_0(s)+R_0(s)+\Sigma_N(s)+I_N(s)
\end{align}
for $\Re s>2+2\varepsilon$,
where
\begin{align*}
&\Sigma_1(s)=-\sum_{b_2\in\mathbb{Z}}\frac{\Gamma(s-\rho(b_2))\Gamma(\rho(b_2))}
{\Gamma(s)}\frac{\zeta_{K_1}^{\prime}}{\zeta_{K_1}}(s-\rho(b_2)),\\
&\Sigma_{1/2}(s)=-\sum_{j=1}^{2g_2}\sum_{c_{j2}\in\mathbb{Z}}
\frac{\Gamma(s-\rho(c_{j2}))\Gamma(\rho(c_{j2}))}
{\Gamma(s)}\frac{\zeta_{K_1}^{\prime}}{\zeta_{K_1}}(s-\rho(c_{j2})),\\
&\Sigma_0(s)=-\sum_{a_2\in\mathbb{Z}\setminus\{0\}}\frac{\Gamma(s-\rho(a_2))\Gamma(\rho(a_2))}
{\Gamma(s)}\frac{\zeta_{K_1}^{\prime}}{\zeta_{K_1}}(s-\rho(a_2)),\\
&\Sigma_N(s)=\sum_{n=1}^{N-1} \binom{-s}{n}
\frac{\zeta_{K_1}^{\prime}}{\zeta_{K_1}}(s+n)
\frac{\zeta_{K_2}^{\prime}}{\zeta_{K_2}}(-n),
\end{align*}
and
$$
I_N(s)=\frac{1}{2\pi\sqrt{-1}}\int_{(-N+\varepsilon)}\frac{\Gamma(s-z)\Gamma(z)}{\Gamma(s)}
\frac{\zeta_{K_1}^{\prime}}{\zeta_{K_1}}(s-z)\frac{\zeta_{K_2}^{\prime}}
{\zeta_{K_2}}(z)dz.
$$

\section{Proof of theorems}\label{sec-5}

Now we consider how to continue the right-hand side of \eqref{4-3} meromorphically
to some wider region.   The (possible) poles of the terms on the right-hand side are given as
follows (where $a_i, b_i,c_{ji}\in\mathbb{Z}$, $1\leq j\leq 2g_i$).

$\bullet$ The poles of $\Sigma_1(s)$: $s=\rho(b_2)-n$ ($n\in\mathbb{N}_0$), 
$\rho(a_1)+\rho(b_2)$, $\rho(b_1)+\rho(b_2)$, and $\rho(c_{j1})+\rho(b_2)$.

$\bullet$ The poles of $\Sigma_{1/2}(s)$: $s=\rho(c_{j2})-n$ ($n\in\mathbb{N}_0$), 
$\rho(a_1)+\rho(c_{j2})$, $\rho(b_1)+\rho(c_{j2})$, and $\rho(c_{j1})+\rho(c_{j2})$.

$\bullet$ The poles of $\Sigma_0(s)$: $s=\rho(a_2)-n$ ($n\in\mathbb{N}_0$), 
$\rho(a_1)+\rho(a_2)$, $\rho(b_1)+\rho(a_2)$, and $\rho(c_{j1})+\rho(a_2)$.

$\bullet$ The poles of $\Sigma_N(s)$: $s=\rho(a_1)-n, \rho(b_1)-n, \rho(c_{j1})-n$
($1\leq n\leq N-1$).

$\bullet$ The poles of $R_0(s)$:
The poles of the term 
$(\zeta_{K_1}^{\prime}/\zeta_{K_1})^{\prime}/(\zeta_{K_1}^{\prime}/\zeta_{K_1})(s)$
are coming from the zeros and poles of $(\zeta_{K_1}^{\prime}/\zeta_{K_1})(s)$.
The numerator of \eqref{diff} (for $i=1$) is a polynomial in $u_1$ of order $2g_1+1$.
Denote the roots of this polynomial by $w_k=|w_k|\exp(\sqrt{-1}\arg w_k)$
($1\leq k\leq 2g_1+1$).    Then the zeros of $(\zeta_{K_1}^{\prime}/\zeta_{K_1})(s)$
are of the form
\begin{align}\label{zeros-d}
\rho(d_{k1})=&\frac{-1}{\log q_1}(\log|w_k|+\sqrt{-1}(\arg w_k+2d_{k1}\pi))\\
&\qquad\qquad (d_{k1}\in\mathbb{Z},1\leq k\leq 2g_1+1).\notag
\end{align}
The poles of $(\zeta_{K_1}^{\prime}/\zeta_{K_1})(s)$ are coming from the zeros and
the poles of $\zeta_{K_1}(s)$, which we have already known.
Therefore the poles of $R_0(s)$ are $s=-n (n\in\mathbb{N}_0), \rho(a_1), \rho(b_1),
\rho(c_{j1})$, and $\rho(d_{k1})$.

\begin{proof}[Proof of Theorem \ref{thm-1}]
Now we consider the case $p_1=p_2$, which we denote by $p$.
The above list implies that the right-hand side of \eqref{4-3} has infinitely many
(possible) poles, but they are distributed discretely.    In fact, in the 
horizontal direction they are located periodically with period $1$,  
while in the vertical direction, since $\log q_i=r_i\log p$ ($i=1,2$), the poles are
distributed periodically with period $2\pi/r_1\log p$ or $2\pi/r_2\log p$ or
$2\pi/[r_1,r_2]\log p$ (where $[r_1,r_2]$ is the minimal common multiple of
$r_1$ and $r_2$).    These poles are therefore not obstacles when we consider the
analytic continuation.

The terms $\Sigma_1(s)$, $\Sigma_{1/2}(s)$, and $\Sigma_0(s)$ are infinite series, 
but in view of the periodicity of $\zeta_{K_1}^{\prime}/\zeta_{K_1}$ in the vertical
direction and Stirling's formula, it is easy to see that these infinite series are
convergent absolutely, and uniformly in any compact subset of $\mathbb{C}$ which
does not include the poles.

Lastly, the integrand of $I_N(s)$ does not have poles if $\Re s>1-N+\varepsilon$.
Therefore we may continue $I_N(s)$ holomorphically to the region
$\Re s>1-N+\varepsilon$, and in this region $\Phi_2^{FF}(s)$ is meromorphic.
Since $N$ is arbitrary, we complete the proof of Theorem \ref{thm-1}.
\end{proof}

We encounter a totally different situation when $p_1\neq p_2$.
In the above list of the poles, the real part of the poles of the form
$\rho(b_1)+\rho(b_2)$ is $2$.   Let
$$
\mathcal{K}=\{\rho(b_1)+\rho(b_2)\mid b_1,b_2\in\mathbb{Z}\}.
$$
We first prove:

\begin{lem}\label{lem-1}
When $p_1\neq p_2$,
the set $\mathcal{K}$ is dense in the vertical line $\{s\mid \Re s=2\}$.
\end{lem}

\begin{proof}
Since
\begin{align*}
\rho(b_1)+\rho(b_2)&=2+2\pi\sqrt{-1}\left(\frac{b_1}{r_1\log p_1}+
\frac{b_2}{r_2\log p_2}\right)\\
&=2+\frac{2\pi\sqrt{-1}}{r_1\log p_1}\left(b_1+b_2\cdot\frac{r_1\log p_1}{r_2\log p_2}\right),
\end{align*}
it is enough to show that the set
$$
\mathcal{K}'=\left\{\left.b_1+b_2\cdot\frac{r_1\log p_1}{r_2\log p_2}\;
\right| \;b_1,b_2\in\mathbb{Z}\right\}
$$
is dense in $\mathbb{R}$.   Since $p_1\neq p_2$, the number
$(r_1\log p_1)/(r_2\log p_2)$ is irrational.   Therefore the set
$\{b_2\cdot(r_1\log p_1)/(r_2\log p_2)\mid b_2\in\mathbb{Z}\}$ (mod 1) is dense in
$[0,1)$, and hence $\mathcal{K}'$ is dense in $\mathbb{R}$.
\end{proof}

Here we quote the following lemma due to Gel'fond \cite{Ge35}:

\begin{lem}\label{lem-Ge}
Let $\alpha_1,\alpha_2$ be non-zero algebraic numbers with heights at most $A$, 
$\beta_1,\beta_2$ be algebraic 
numbers with heights at most $B (\geq 2)$, and 
$\Lambda=\beta_1\log\alpha_1+\beta_2\log\alpha_2$.   
Then either $\Lambda=0$ or $|\Lambda|>B^{-C}$,  where $C>0$ depends only on degrees of
$\alpha_i$s, $\beta_i$s and $A$.
\end{lem}

This lemma is now a special case of a more general theorem; see \cite[Theorem 3.1]{Ba75}.

Using Lemma \ref{lem-Ge} we obtain the following lemma, which is the key fact in 
the proof of Theorem \ref{thm-2}.

\begin{lem}\label{lem-2}
For any element $\rho(b_1^0)+\rho(b_2^0)\in\mathcal{K}$, the sum $\Sigma_1(s)$ tends to
infinity as $s$ tends to $\rho(b_1^0)+\rho(b_2^0)$ from the right.
\end{lem}

\begin{proof}
From \eqref{diff}, for any $b_2\in\mathbb{Z}$ we have
\begin{align*}
&\frac{\zeta_{K_1}^{\prime}}{\zeta_{K_1}}(s-\rho(b_2))\\
&\quad=
  -u_1(s-\rho(b_2))\log q_1\frac{dL_{K_1}(u_1(s-\rho(b_2)))/du_1(s-\rho(b_2))}
  {L_{K_1}(u_1(s-\rho(b_2)))}\notag\\
&\qquad-\frac{1+q_1-2q_1u_1(s-\rho(b_2))}{(1-u_1(s-\rho(b_2)))
(1-q_1u_1(s-\rho(b_2)))}\notag\\
&\quad =B^*(s,b_2)+B^{**}(s,b_2),\notag
\end{align*}
say, where $u_1(s-\rho(b_2))=q_1^{-s+\rho(b_2)}$.
Accordingly we divide
\begin{align}\label{5-0}
\Sigma_1(s)&=-\sum_{b_2\in\mathbb{Z}}\frac{\Gamma(s-\rho(b_2))\Gamma(\rho(b_2))}
{\Gamma(s)}B^*(s,b_2)\\
&\quad-\sum_{b_2\in\mathbb{Z}}\frac{\Gamma(s-\rho(b_2))\Gamma(\rho(b_2))}
{\Gamma(s)}B^{**}(s,b_2)\notag\\
&=-\Sigma_1^*(s)-\Sigma_1^{**}(s),\notag
\end{align}
say.
Put $s=\rho(b_1^0)+\rho(b_2^0)+\eta$, with a small positive number $\eta$.
Then $u_1(s-\rho(b_2))=q_1^{-\rho(b_1^0)-\rho(b_2^0)+\rho(b_2)-\eta}$, 
whose absolute value is $q_1^{-1-\eta}$, 
so we see that $L_{K_1}(u_1(s-\rho(b_2)))\neq 0$, and hence
$B^*(s,b_2)$ remains bounded as $\eta\to 0$.
Therefore, again noting Stirling's formula we see that 
$\Sigma_1^*(s)$ remains finite as $\eta\to 0$.

Consider $B^{**}(s,b_2)$.    When $b_2=b_2^0$, we have 
$u_1(s-\rho(b_2^0))=q_1^{-\rho(b_1^0)-\eta}$, and so
$$
1-q_1 u_1(s-\rho(b_2^0))=1-q_1^{-2b_1^0\pi\sqrt{-1}/\log q_1 -\eta}
=1-q_1^{-\eta}
=\eta\log q_1+O(\eta^2).
$$
This implies that $B^{**}(s,b_2^0)$ diverges as $\eta\to 0$.

When $b_2\neq b_2^0$, we have
\begin{align}\label{5-1}
&1-q_1 u_1(s-\rho(b_2))=1-q_1^{-\rho(b_1^0)-\rho(b_2^0)+\rho(b_2)-\eta}\\
&\qquad=1-q_1^{2\pi\sqrt{-1}(b_2-b_2^0)/\log q_2-\eta}\notag\\
&\qquad=1-\exp\left(2\pi\sqrt{-1}(b_2-b_2^0)\frac{\log q_1}{\log q_2}\right)q_1^{-\eta}
\notag\\
&\qquad=1-\exp\left(2\pi\sqrt{-1}\left((b_2-b_2^0)\frac{\log q_1}{\log q_2}-m\right)\right)q_1^{-\eta},\notag
\end{align}
where $m$ is the integer nearest to $(b_2-b_2^0)\log q_1/\log q_2$.
Let
$$
X=(b_2-b_2^0)\frac{\log q_1}{\log q_2}-m
= \frac{1}{\log q_2}((b_2-b_2^0)\log q_1-m\log q_2).
$$
Since now $b_2\neq b_2^0$, we have $X\neq 0$.    Therefore,
applying Lemma \ref{lem-Ge} to the above with $\alpha_1=q_1$, $\alpha_2=q_2$, 
$\beta_1=b_2-b_2^0$ and $\beta_2=-m$, we obtain
$$
|X|>\frac{1}{\log q_2}B^{-C}
$$
with $C=C(q_1,q_2)>0$, where $B$ is the maximum of the heights of $b_2-b_2^0$ and
$m$.
Since $m\leq (b_2-b_2^0)\log q_1/\log q_2+1$, we find that
$$
B\leq (|b_2|+|b_2^0|)\left(1+\frac{\log q_1}{\log q_2}\right)+1,
$$
and hence
\begin{align}\label{5-2}
|X| \gg_{q_1,q_2} (|b_2|+|b_2^0|)^{-C}.
\end{align}
The right-hand side of \eqref{5-1} is
\begin{align}\label{5-3}
=1-\exp(2\pi\sqrt{-1}X)q_1^{-\eta}=q_1^{-\eta}(q_1^{\eta}-\exp(2\pi\sqrt{-1}X)).
\end{align}
Since $q_1^{\eta}\geq 1$, we find that
$$
|q_1^{\eta}-\exp(2\pi\sqrt{-1}X)|\geq |1-\exp(2\pi\sqrt{-1}X)|,
$$
so \eqref{5-3} (and hence \eqref{5-1}) is 
$$
\gg q_1^{-\eta}|1-\exp(2\pi\sqrt{-1}X)|\gg_{q_1}|X|.
$$
Combining this estimate with \eqref{5-2}, we obtain
\begin{align}\label{5-4}
|1-q_1u_1(s-\rho(b_2))|\gg_{q_1,q_2}(|b_2|+|b_2^0|)^{-C}.
\end{align}
Therefore the total contribution of the terms $b_2\neq b_2^0$ to $\Sigma_1^{**}(s)$ is
\begin{align*}
\ll \sum_{b_2\neq b_2^0}\left|\frac{\Gamma(s-\rho(b_2))\Gamma(\rho(b_2))}
{\Gamma(s)}\right|\cdot (|b_2|+|b_2^0|)^C,
\end{align*}
which is convergent absolutely by virtue of Stirling's formula.

The conclusion is that, on the right-hand side of \eqref{5-0}, only the term 
in $\Sigma_1^{**}(s)$ corresponding to $b_2^0$ is divergent as $\eta\to 0$, while
the other sums remain bounded, which implies the assertion of the lemma.
\end{proof}

\begin{remark}\label{key-ineq}
In the above proof, a key inequality is \eqref{5-2}, which may be regarded as the
function-field analogue of Conjecture (B) mentioned in Section \ref{sec-1}.
\end{remark}

\begin{proof}[Proof of Theorem \ref{thm-2}]
Now we complete the proof of Theorem \ref{thm-2}.
Lemma \ref{lem-1} and Lemma \ref{lem-2} imply that $\mathcal{K}$ is dense in
$\{s\mid \Re s=2\}$, and all points of $\mathcal{K}$ are singularities of
$\Sigma_1(s)$.    That is, $\{s\mid \Re s=2\}$ is the natural boundary of
$\Sigma_1(s)$.
On the other hand, it is easy to see that $\Sigma_{1/2}(s)$, $\Sigma_0(s)$, $\Sigma_N(s)$ 
have no pole on $\{s\mid \Re s=2\}$.

Therefore, if we can see that $R_0(s)$ also has no pole on $\{s\mid \Re s=2\}$,
then the set of singularities of $\Sigma_1(s)$ on $\{s\mid \Re s=2\}$ is exactly the
set of singularities of $\Phi_2^{FF}(s)$, and hence the conclusion follows.

So far we cannot exclude the possibility of the existence of poles of $R_0(s)$
on $\{s\mid \Re s=2\}$, but even if so, those poles would distribute discretely
(see Remark \ref{rem-R_0} below), 
and hence it does not affect the conclusion.
 \end{proof}

\begin{remark}\label{rem-R_0}
The term $R_0(s)$ has poles of the form $\rho(d_{k1})$, whose real part is given
by $-(\log |w_k|)/(\log q_1)$ (see \eqref{zeros-d}).

(i) The distribution of the roots $w_k$ has not been well studied, so at present
we cannot exclude the possibility that $|w_k|=q_1^{-2}$ for some $k$.   If so,
then the corresponding $\rho(d_{k1})$ are on $\{s\mid \Re s=2\}$.    But these
are discretely distributed; in fact, they are distributed periodically with
period $2\pi\sqrt{-1}/\log q_1$.

(ii) If $|w_k|< q_1^{-2}$ for some $k$, then the real part of the corresponding
$\rho(d_{k1})$ is larger than $2$.
Hence $\Phi_2^{FF}(s)$ has poles in the region $\Re s>2$, but this contradicts
Proposition \ref{prop-1}.    Therefore, as a by-product of our theory, we find that
$|w_k|\geq q_1^{-2}$ for all $k$. 
\end{remark}


\end{document}